\documentclass[12pt,a4paper,psamsfonts]{amsart}
\usepackage{amssymb,amscd,amsxtra,calc}
\usepackage{cmmib57}

\theoremstyle{plain}
    \newtheorem{thm}{Theorem}[section]

    \newtheorem{corollary}[thm]{Corollary}
    \newtheorem{lemma}[thm]{Lemma}
    \newtheorem{proposition}[thm]{Proposition}
    \newtheorem{question}[thm]{Question}
    \newtheorem{theorem}[thm]{Theorem}
    \newtheorem{problem}[thm]{Problem}

\theoremstyle{definition}
    \newtheorem{remark}[thm]{Remark}
    \newtheorem{setup}[thm]{}

\newcommand{\bG}{{\mathbb{G}}}
\newcommand{\bP}{{\mathbb{P}}}

\newcommand{\Gbar}{{\overline{G}}}
\newcommand{\Hbar}{{\overline{H}}}
\newcommand{\Mbar}{{\overline{M}}}
\newcommand{\Gtilde}{{\tilde{G}}}
\newcommand{\Htilde}{{\tilde{H}}}
\newcommand{\Mtilde}{{\tilde{M}}}
\newcommand{\Aut}{{\operatorname{Aut}}}
\newcommand{\Dom}{{\operatorname{Dom}}}
\newcommand{\id}{{\operatorname{id}}}
\newcommand{\GL}{{\operatorname{GL}}}
\newcommand{\Image}{{\operatorname{Im}}}
\newcommand{\Lie}{{\operatorname{Lie}}}
\newcommand{\SL}{{\operatorname{SL}}}
\newcommand{\Sp}{{\operatorname{Sp}}}
\newcommand{\SO}{{\operatorname{SO}}}
\newcommand{\PGL}{{\operatorname{PGL}}}
\newcommand{\Gm}{{{\bG}_{\textup{m}}}}
\newcommand{\Ga}{{{\bG}_{\textup{a}}}}

\newcommand{\bksl}{{\backslash}}
\newcommand{\characteristic}{{\operatorname{char}}}

\begin{document}

\title
{Rationality of homogeneous varieties}

\author{CheeWhye Chin}
\address
{Department of Mathematics \endgraf
 National University of Singapore \endgraf
 10 Lower Kent Ridge Road, Singapore 119076, Singapore}
\email{cheewhye@nus.edu.sg}

\author{De-Qi Zhang}
\address
{Department of Mathematics \endgraf
 National University of Singapore \endgraf
 10 Lower Kent Ridge Road, Singapore 119076, Singapore}
\email{matzdq@nus.edu.sg}

\begin{abstract}
Let $G$ be a connected linear algebraic group
over an algebraically closed field $k$,
and let $H$ be
a connected closed subgroup of $G$.
We prove that
the homogeneous variety $G/H$
is a rational variety over $k$
whenever $H$ is solvable,
or when $\dim(G/H) \leqslant 10$
and $\characteristic(k)=0$.
When $H$ is of maximal rank in $G$,
we also prove that
$G/H$ is rational
if the maximal semisimple quotient of $G$ is
isogenous to
a product of almost-simple groups of
type~$A$,
type~$C$ (when $\characteristic(k) \neq 2$),
or
type~$B_3$ or $G_2$ (when $\characteristic(k) = 0$).
\end{abstract}


\thanks{}

\maketitle

\tableofcontents


\section{Introduction}

Let $k$ be
an algebraically closed field
of arbitrary characteristic $p \ge 0$.
%
Throughout this paper,
we work only with the Zariski topology
on varieties or algebraic groups over $k$.
By a result that goes back to Chevalley
when $\characteristic(k)=0$
(cf.~\cite[\S2, Cor.~2 to Th.~1]{Ch})
and to Rosenlicht
for arbitrary $\characteristic(k)$
(cf.~\cite[end of \S3]{Ro57}),
a connected linear algebraic group $G$ over $k$
is a rational variety:
the field $k(G)$ of rational functions on $G$
is a purely transcendental extension of $k$.
If $H \subseteq G$ is any closed subgroup,
the homogeneous variety $G/H$
is thus unirational:
its field $k(G/H)$ of rational functions
is contained in
a purely transcendental extension of $k$
(namely, $k(G)$).
It is thus natural to consider
the following:

\begin{problem}
\label{RatPbm}
Let $G$ be a connected linear algebraic group
over the algebraically closed field $k$,
and let $H \subseteq G$ be a closed subgroup.
Is the homogeneous variety $G/H$ rational?
Equivalently:
is the field $k(G/H)$ of rational functions on $G/H$
a purely transcendental extension of $k$?
\end{problem}

This is a long-standing question
in algebraic geometry,
known as the \emph{rationality problem}.
When $H$ is not assumed to be connected,
the answer is negative in general;
see remark~\ref{rmkSaltman} below.
When $H$ is connected,
the rationality problem
in the generality of~\ref{RatPbm}
is open
even when $\characteristic(k)=0$.
The problem was mentioned
(possibly for the first time)
in~\cite{Hab},
in which it was suggested that
the answer is negative in general
even when $H$ is connected,
although no counter-example
is known to date.
However,
affirmative answers have been established
in several cases;
for instance,
the rationality of $G/H$
when $\dim(G/H) \leqslant 4$
and $\characteristic(k)=0$
was proved in~\cite[Lemma 1.15]{MU}.

A more general form of
the rationality problem
is concerned with
the rationality of the field $k(X)^H$
when $X$ itself is a rational variety
on which a linear group $H$ acts.
Two typical versions of this problem
have been considered in the literature:
namely,
when $X$ is the underlying vector space of
a finite dimensional representation of $H$,
and when $X$ is the underlying variety of
a connected group $H$
acting on itself by conjugation;
see \cite{CTS} for a survey of the former,
and \cite{CTKPR} for works on the latter.
The variant~\ref{RatPbm} of the rationality problem
we consider in this paper
amounts to the case
when $X$ is the underlying variety of
a connected group $G$
on which a subgroup $H$ acts
by right multiplication.
Our goal is to give
several more criteria
under which
one can establish
an affirmative answer
to this variant of
the rationality problem,
and to the extent possible,
in a characteristic-free manner.

\begin{setup}
\quad
Here, we give an overview of
the main results of this paper.
We refer the reader to
theorems~\ref{Th0}, \ref{ThA}, \ref{ThB0} and \ref{ThB}
for the precise statements.

In theorem~\ref{Th0},
we show that
$G/H$ is rational
for any closed subgroup $H$
contained in a Borel subgroup $B$ of $G$;
in particular,
this is so whenever
$G$ or $H$ is connected and solvable
(for then
$H$ is necessarily contained in
some Borel subgroup of $G$).
This result holds
for any characteristic of $k$
and is probably folklore
to the experts,
although it seems not to have appeared
in the literature.
It is analogous to
a theorem of Miyata (cf.~\cite{Mi})
asserting that
the field $k(V)^H$ is
a purely transcendental extension of $k$
when $H$ is a linear group
and $V$ is a finite dimensional representation of $H$
which is triangularizable.
We establish theorem~\ref{Th0}
via a certain splitting principle
for the quotient map $G \dasharrow G/B$
(cf.~cor.~\ref{solv split})
which was ``morally speculated'' by Prof.~V.~Popov
in a conversation with the second author.
The main tools needed are
some classical results of Rosenlicht
(cf.~lemmas~\ref{tSS} and \ref{Rosenlict generic section}).

In theorem~\ref{ThA},
we show that
$G/H$ is rational
for any connected closed subgroup $H$
of maximal rank in $G$,
if the maximal semisimple quotient of $G$
is isogenous to
a product of almost-simple groups of
type~$A$,
type~$C$ (when $\characteristic(k) \neq 2$),
or
type~$B_3$ or $G_2$ (when $\characteristic(k) = 0$).
Here, the main ingredient is
the Borel-de~Siebenthal algorithm
(cf.~\cite[Th.~6]{BdS}
and the algebraic group version
in~\cite[prop.~6.6]{Le})
classifying
maximal connected reductive subgroups
of maximal rank
in a given connected semisimple group.
We also employ
general structural results of algebraic groups
such as the Bruhat decomposition
(cf.~lemma~\ref{parabolic section}),
the theorem of Borel-Tits
(cf.~lemma~\ref{Borel-Tits}),
as well as
properties of special groups
(in the sense of Serre).
Our theorem~\ref{ThA} may be compared
with~\cite[Th.~0.2]{CTKPR},
where the authors prove, among other things, that
for an almost-simple group $G$ of type~$A$ or~$C$,
acting on itself by conjugation,
the field extension $k(G)/k(G)^{G}$
is purely transcendental.

In sections~\ref{low dim homog} and~\ref{low dim homog ctd},
we assume $\characteristic(k)=0$
and use geometric arguments
(cf.~lemma~\ref{finbir})
to show the rationality of certain
low dimensional homogeneous varieties $G/H$,
our starting point being
the classical theorems of L\"uroth and Castelnuovo
(cf.~lemma~\ref{dim1or2})
asserting the rationality of
unirational varieties of dimension~$\leqslant 2$
(over our algebraically closed field $k$ of characteristic zero).
We show in theorem~\ref{ThB0} that
$G/H$ is rational
whenever $\dim(G/H) \leqslant 5$,
and in theorem~\ref{ThB} that
if $H$ is connected,
$G/H$ is rational
whenever $\dim(G/H) \leqslant 10$.
We follow the approach in~\cite[Lemma 1.15]{MU}
of reducing to the case when $G$ is semisimple,
but we utilize new ingredients
such as the geometry of the big cell
and the structure of the centralizers of subtori
in a connected reductive group
(cf.~\cite[\S 28.5, \S 22.4, \S 26.2]{Hu}).

\end{setup}

\begin{remark}
\label{rmkSaltman}
Our results provide
affirmative answers to
the rationality problem~\ref{RatPbm}
in the respective situations,
but except for
low dimensional homogeneous varieties
(i.e.~when $\dim(G/H) \leqslant 5$
as in theorem~\ref{ThB0}),
we require that
the closed subgroup
$H \subseteq G$
be connected.
This is not surprising,
since there are examples of
finite subgroups $H$
in a product $G$ of general linear groups
for which
$G/H$ is not rational.
Indeed,
by the work of Saltman \cite[Th.~3.6]{Sa}
in the context of
the Noether problem,
one knows that
for each prime~$p$,
there exist finite $p$-groups $H$
with the following property:
for any algebraically closed field $k$
with $\characteristic(k)\neq p$,
if $V$ denotes
the vector space of
the regular representation of $H$
over $k$,
and if $k(V)^H$ denotes
the subfield of $H$-invariants
in the field $k(V)$ of
rational functions on $V$ over $k$,
then $k(V)^H$ is not
a purely transcendental extension of $k$,
so the variety $V/H$ is not rational.
If we decompose
$V \cong \bigoplus_i V_i^{\oplus m_i}$
into irreducible representations $V_i$ of $H$,
then $m_i = \dim_k V_i$,
and the linear representation of $H$
on the isotypic component
$V_i^{\oplus m_i}$
yields an action of $H$
by right multiplication on
$G_i := \GL_{m_i}(k)$.
Thus
$V$ contains an $H$-stable open subvariety
isomorphic to
$G := \prod_i G_i$.
Since $V/H$ is not rational,
$G/H$ is not rational either.
\end{remark}

\begin{setup}
\quad
As we are mainly interested in
the birational properties of homogeneous varieties,
we consider only rational actions
and rational quotients
(unless explicitly stated otherwise).
A (right) \emph{rational action} of
a linear algebraic group $H$
on an irreducible variety $X$
(cf.~\cite[\S2.3]{BSU})
is a group homomorphism
$H \to \Aut(k(X))$
from $H$ to the group $\Aut(k(X))$ of
birational automorphisms of $X$,
such that the resulting rational map
$a : X \times H \dasharrow X$
is (defined and) regular on
an open dense subset of $X \times H$.
Given such a rational action,
the \emph{rational quotient of $X$ by $H$}
(in the sense of Rosenlicht,
cf.~\cite[Th.~2]{Ro56})
is any irreducible variety
whose field of rational functions
is identified with $k(X)^H$;
i.e.~it is a $k$-variety $X/H$
characterized up to birational equivalence
by the equality
$k(X/H) = k(X)^H$
of rational function fields.
\footnote{
Given a rational action of $H$ on $X$,
we may replace $X$
by another birational model
and assume that
$H$ acts on $X$ regularly
(cf.~\cite[Th.~1]{Ro56}),
and it then follows
(essentially by~\cite[Th.~2]{Ro56})
that
there exists
a geometric quotient of
(an open subset of) $X$ by $H$
(in the sense of GIT);
the rational quotient $X/H$
can thus be regarded as
the birational equivalence class of
the moduli space of
``general $H$-orbits in $X$''.
}
The $k$-inclusion
$k(X/H) \subseteq k(X)$ of fields
induces
a dominant rational map
$X \dasharrow X/H$,
called the \emph{rational quotient map},
which is $H$-equivariant
with respect to
the trivial action on $X/H$,
and has the universal property that
any dominant rational map from $X$
which is constant on
general $H$-orbits in $X$
factors through it.
Thus,
if $G$ is a connected linear algebraic group
and $H \subseteq G$ is a closed subgroup,
the homogeneous space $G/H$ of $H$-cosets in $G$
is regarded as the rational quotient of
the variety $G$
by the right multiplication action of $H$.
Similarly,
if $K \subseteq G$ is another closed subgroup,
the space of double cosets
$K \bksl G / H$
is both the rational quotient of
the variety $K \bksl G$
by the right multiplication action of $H$,
as well as
the rational quotient of
the variety $G/H$
by the left multiplication action of $K$.

\end{setup}

\bigskip

\noindent
{\bf Acknowledgement.}

The authors heartily thank
our friends and colleagues,
and especially the referee of
an earlier version of this paper,
for the remarks~\ref{rmkSaltman} and~\ref{rmkBGBHrational},
and also
prop.~\ref{PropB}(d),
which allows the case of type~$C$
in theorem~\ref{ThA}
to be included.
The second author would like to offer his thanks
to M.~Brion
for valuable discussions about the rationality of $G/H$
when $G$ is solvable
and
for reminding him of the reference to \cite{GIT}
in the argument for Lemma~\ref{Borel-Tits}.
He also thanks M.~Chen for the warm hospitality
during his visit to Fudan University
while working on this paper.
He is partially supported by an ARF of NUS.

%
%
%
%


\section{Quotients by subgroups contained in a Borel}

In this section,
we work over
an algebraically closed base field $k$
of arbitrary characteristic,
and we consider
the rationality problem~\ref{RatPbm}
for $G/H$
when $H$ is contained in
a Borel subgroup of $G$.
Our goal is to establish theorem~\ref{Th0}.
We first review
some classical results and arguments.

\begin{lemma}
[Rosenlicht]
\label{tSS}
Let $H$ be a linear algebraic group
acting rationally on a variety $X$,
and let $K \lhd H$ be
a closed normal subgroup.
Then $H' := H/K$ acts rationally on
the quotient variety
$X' := X/K$,
and $X'/H'$ is (naturally) birational to $X/H$.
\end{lemma}

See~\cite[Th.~5]{Ro56}.
In particular,
if $K \lhd H$ lies in the kernel of
the action of $H$ on the field $k(X)$
(i.e.~$K$ acts trivially on
an open dense subset of $X$),
then
$X' = X/K$ is (birational to) $X$ itself,
and hence
$X/H$ is birational to $X/H'$.
In other words,
in forming the rational quotient $X/H$,
we may replace $H$ by
its image in the birational automorphism group
$\Aut(k(X))$ of $X$.
This is also clear from
the characterization of
the rational quotient
by its rational function field.

Let $H$ be a linear algebraic group
acting rationally on a variety $X$.
The action is called \emph{generically free}
iff there exists
an open dense subset
$U \subseteq X$
such that
for every $x \in U$,
the stabilizer subgroup
${}_xH := \{h \in H \,|\, x \cdotp h = x\}$
is trivial.
An easy example of
a generically free action
is the right multiplication action of $H$
on a connected linear algebraic group $G$
which contains $H$ as a closed subgroup.
A generically free action
is necessarily
\emph{generically faithful},
i.e.~the homomorphism
$H \to \Aut(k(X))$
is injective;
but the converse does not hold in general.
However,
one has the following:

\begin{lemma}
\label{Demazure generic freeness}
Suppose $H$ is a torus
acting rationally on a variety $X$.
If the action is
generically faithful,
then it is
generically free.
In particular,
one has
$\dim X = \dim(X/H) + \dim H$.
\end{lemma}

See~\cite[\S1, Prop.~10]{De}.

\begin{lemma}
\label{spQ}
\footnote{
After the second author gave a talk at Fudan University
on a preliminary version of this paper,
the PhD student JinSong Xu of the host Professor Meng Chen
informed us of an independent proof of this result
in the case when
$X$ is a homogeneous variety.
}
Let $M$ be a linear algebraic group
acting rationally on a variety $X$.
Suppose that
the action of $M$ on $X$ is generically free,
and that
the rational quotient map
$\pi : X \dasharrow X/M$
admits a rational section
$s : X/M \dasharrow X$.
Then there exists
a birational map
\[
  f
  \ :\
  (X/M) \times M
  \dashrightarrow
  X
  \qquad
  \text{which is $M$-equivariant}
\]
with respect to
the natural right action of $M$
on the domain $\Dom(f)$ of $f$
via multiplication on the second factor $M$.
In particular,
for any closed subgroup $H \subseteq M$,
$X/H$ is birational to $(X/M) \times (M/H)$.
\end{lemma}

This is the birational analogue of
the classical statement that
a principal bundle with a section
is trivial.

\begin{proof}
We assume that
$M$ acts on $X$ from the right.
Replacing $X$ by another birational model,
we may assume
(cf.~\cite[Th.~1]{Ro56})
that
the action is biregular,
so that one has a quotient morphism
$\pi : X \to X/M$.
Since the singular locus of $X$
is stable under the action of $M$,
and the action
is generically free,
we may replace $X$
by an open dense subset
and assume that
$X$ is non-singular
and that
the action is in fact free,
i.e.~the stabilizer subgroup
${}_xM$ is trivial
for every $x \in X$.
Likewise,
shrinking $X/M$
(and $X$ correspondingly)
if necessary,
we also assume that
$X/M$ is non-singular.
Further replacing $X/M$ by
the domain $\Dom(s)$
of the rational section~$s$
and $X$ by $\pi^{-1}(\Dom(s))$,
we may assume that
$s : X/M \to X$ is a regular section,
i.e.~a morphism
such that
$\pi \circ s = \id_{X/M}$.
Thus we have a bijective morphism
\[
  f
  \ :\
  (X/M) \times M
  \longrightarrow
  X
  ,
  \qquad
  (y,m)
  \mapsto
  s(y) \cdotp m
\]
which is clearly $M$-equivariant:
$f(y,m) \cdotp m'
 =
 f(y, m \cdotp m')
$.
Since $X/M$ and $X$
are normal varieties,
we may apply
\cite[Lemma 6.14(ii)]{Bo}
to infer that
$f$ is in fact
an isomorphism.
This shows the main assertion
of the lemma.
The final claim follows
by passing to the quotient by $H$.
\end{proof}

\begin{lemma}
[Rosenlicht's generic section theorem]
\label{Rosenlict generic section}
Let $M$ be a connected solvable group
acting rationally on a variety $X$.
Then
the rational quotient map
$\pi : X \dasharrow X/M$
admits a rational section
$s : X/M \dasharrow X$.
\end{lemma}

See~\cite[Th.~10]{Ro56}.

\begin{corollary}
\label{solv split}
Let $G$ be a connected linear algebraic group,
and let $H \subseteq G$ be any closed subgroup
contained in a Borel subgroup $B$ of $G$.
Then
$G/H$ is birational to $(G/B) \times (B/H)$.
\end{corollary}

\begin{proof}
Since $B$ is connected and solvable,
the quotient map $G \dasharrow G/B$
has a rational section
by lemma~\ref{Rosenlict generic section}.
The right-multiplication action of $B$ on $G$
is generically free,
so lemma~\ref{spQ} is applicable
and yields the corollary.
\end{proof}

\begin{lemma}
\label{Sol}
Let $M$ be a connected solvable group
acting rationally on a variety $X$.
Then
$X$ is birational to
$(X/M) \times \bP^d$
for some $d \leqslant \dim M$.
In particular,
if $X/M$ is rational,
then so is $X$.
\end{lemma}

\begin{proof}
Since $M$ is connected solvable,
we can find
a sequence of connected normal closed subgroups
$M = M_r \rhd M_{r-1} \rhd \cdots \rhd M_1 \rhd M_0 = \{1\}$
such that
the subquotients $M_{i+1}/M_i$
are 1-dimensional groups,
isomorphic to either $\Gm$ or $\Ga$.
For each $i \in \{0,\ldots,r{-}1\}$,
we will show that
$X/M_i$ is birational to
$X/M_{i+1}$ or $(X/M_{i+1}) \times \bP^1$;
by descending induction on $i$,
it would then follow that
$X = X/M_0$ is birational to
$(X/M_r) \times \bP^d$
for some $d \leqslant r = \dim M$,
whence the lemma.

Consider the rational action of
$H_i := M_{i+1}/M_i$
on the variety $X_i := X/M_i$,
and let $H_i'$ denote the image of $H_i$
in $\Aut(k(X_i))$;
by lemma~\ref{tSS},
$X/M_{i+1}$ is naturally birational to $X_i/H_i'$,
and
$H_i'$ is either trivial
or isomorphic to $\Gm$ or $\Ga$.
If $H_i' \cong \Gm$,
by lemma~\ref{Demazure generic freeness},
the action of $H_i'$ on $X_i$
is generically free,
and
by lemma~\ref{Rosenlict generic section},
the rational quotient map
$X_i \dasharrow X_i/H_i'$
admits a rational section
$X_i/H_i' \dasharrow X_i$,
whence by lemma~\ref{spQ},
$X_i$ is birational to $(X_i/H_i') \times H_i'$.
If $H_i' \cong \Ga$,
we replace $X$
by a suitable birational model
and assume that
it is affine,
and then apply~\cite[Prop.~14.2.2]{Sp}
to see that
$X_i$ is birational to $(X_i/H_i') \times H_i'$,
which proves what we want.
\end{proof}

\begin{corollary}
\label{Sol quotient}
Let $M$ be a connected solvable group.
Then for any closed subgroup
$H \subseteq M$,
the quotient variety $M/H$
is rational.
\end{corollary}

\begin{proof}
Apply lemma~\ref{Sol}
to the natural left action of $M$
on $X := M/H$,
and note that
the rational quotient $M \backslash X$
of $X$ by $M$
is a point.
\end{proof}

\begin{remark}
In lemma~\ref{Sol},
the connected solvable group $M$
acts on a variety $X$
which is not necessarily a group;
this mildly generalizes~\cite[Cor.~1 to Th.~10]{Ro56},
and will be very convenient
for us later on.
However,
cor.~\ref{Sol quotient}
which is deduced from lemma~\ref{Sol}
does not give
the best result:
the quotient variety $M/H$ there
is in fact
isomorphic to
a product of copies of
$\Ga$ and $\Gm$;
see~\cite[Theorem 5]{Ro63}.
Both results,
as well as
lemma~\ref{Rosenlict generic section},
hold for split connected solvable
linear algebraic groups
over an arbitrary base field,
by essentially
the same argument.
\end{remark}

\begin{theorem}
\label{Th0}
Let $G$ be a connected linear algebraic group,
and let $H \subseteq G$ be any closed subgroup
contained in a Borel subgroup of $G$.
Then
$G/H$ is a rational variety.
\end{theorem}

\begin{proof}
Let $B$ be a Borel subgroup of $G$
containing $H$.
Then $G/H$ is birational to
$(G/B) \times (B/H)$
by cor.~\ref{solv split}.
The quotient flag variety $G/B$
is rational
(well-known; see also
lemma~\ref{parabolic section} below),
and $B/H$ is rational
by cor.~\ref{Sol quotient}.
Hence $G/H$ is rational.
\end{proof}

\bigskip

We record below
some reduction arguments
which will be useful later on.

\begin{lemma}
\label{mod radical}
Let $G$ be a connected linear algebraic group
and let $H \subseteq G$ be any closed subgroup.
Let $R := R(G)$ be
the solvable radical of $G$,
let $G' := G/R$ be
the maximal semisimple quotient of $G$,
and let $H' := H/(H \cap R)$ be
the image of $H$ in $G'$.
Then
$G/H$ is birational to
$G'/H' \times \bP^s$
for some $s \leqslant \dim R$.
In particular,
if $G'/H'$ is rational
then so is $G/H$.
\end{lemma}

\begin{proof}
By lemma~\ref{Sol}
applied to the natural left action of $R$ on $G/H$,
we see that
$G/H$ is birational to
$(G/HR) \times \bP^s$
for some $s \leqslant \dim R$.
The result follows
from the observation that
$G/HR$ is isomorphic to
$(G/R)/(H/H \cap R) = G'/H'$.
\end{proof}

\begin{lemma}
\label{no-name-lemma}
For $i=1,2$,
let $M_i$ be a connected solvable group
acting rationally on a variety $X_i$.
Assume that
$X_1/M_1$ is birational to $X_2/M_2$,
and that
$\dim X_1 \leqslant \dim X_2$.
Then
$X_2$ is birational to $X_1 \times \bP^d$
where $d := \dim X_2 - \dim X_1$.
In particular,
if $X_1$ is rational,
then so is $X_2$.
\end{lemma}

\begin{proof}
By lemma~\ref{Sol},
$X_i$ is birational to
$(X_i/M_i) \times \bP^{d_i}$
for some $d_i \geqslant 0$.
Since
$X_1/M_1$ is birational to $X_2/M_2$
by assumption,
it follows that
$d = d_2 - d_1$
and that
$X_2$ is birational to $X_1 \times \bP^d$.
\end{proof}

%
%
%
%


\section{Quotients by connected subgroups of maximal rank}

We continue to work over
an algebraically closed base field $k$,
of arbitrary characteristic
unless otherwise stated.
Let $G$ be a connected linear algebraic group.
The maximal semisimple adjoint quotient $\Gbar$ of $G$
decomposes as a direct product
$\Gbar = \Gbar_1 \times \cdots \times \Gbar_\ell$
of connected adjoint almost-simple groups;
we refer to these $\Gbar_i$'s
as the \emph{adjoint factors of $G$}.
Our main result in this section is:

\begin{theorem}
\label{ThA}
Let $G$ be a connected linear algebraic group,
and let $H \subseteq G$ be
a connected closed subgroup
of maximal rank in $G$.
Assume that
each adjoint factor of $G$ is of
type~$A$,
type~$C$ (when $\characteristic(k) \neq 2)$,
or
type~$B_3$ or $G_2$ (when $\characteristic(k) = 0)$.
Then
$G/H$ is a rational variety.
\end{theorem}

We prove this in~\ref{ThAproof};
the essential case is when
$G$ is simply-connected
and almost-simple.
To reduce to this case,
let $G_i$ denote
the simply-connected cover of
the adjoint factor $\Gbar_i$,
let $\Hbar_i$ denote
the image of $H$ in $\Gbar_i$,
and let $H_i$ denote
the preimage of $\Hbar_i$ in $G_i$.

\begin{lemma}
\label{reduction to adjoint factors}
With the above notation,
each $H_i$
is a subgroup of maximal rank in $G_i$,
and
$G/H$ is birational to
$(G_1/H_1) \times \cdots \times (G_\ell/H_\ell) \times \bP^s$
for some $s \leqslant \dim R$,
where $R := R(G)$ is
the solvable radical of $G$.
In particular,
if each $G_i/H_i$ is rational
then so is $G/H$.
\end{lemma}

\begin{proof}
Let $G' := G/R$ be
the maximal semisimple quotient of $G$.
The image $H'$ of $H$ in $G'$
is a subgroup of maximal rank in $G'$.
By lemma~\ref{mod radical},
$G/H$ is birational to
$(G'/H') \times \bP^s$
for some $s \leqslant \dim R$.

Next,
let $Z(G')$ be the centre of $G'$,
let $\Gbar := G'/Z(G')$ be
the adjoint quotient of $G'$,
and consider the quotient isogeny
$q : G' \to \Gbar$.
The map
$\Hbar \mapsto H' := q^{-1}(\Hbar)$
is a bijection between
the collection of
connected subgroups of maximal rank of $\Gbar$
and
the collection of
connected subgroups of maximal rank of $G'$
(cf.~\cite[Exp.~XII Cor.~7.12]{SGA3-2}).
Since $H'$ contains $Z(G')$,
one obtains an isomorphism between
$G'/H'$ with $\Gbar/\Hbar$.
Applying the same argument
to the quotient isogeny
$\Gtilde \to \Gbar$
from the simply-connected cover $\Gtilde$ of $\Gbar$,
one obtains an isomorphism between
$\Gtilde/\Htilde$ with $\Gbar/\Hbar$,
where $\Htilde$ is
the preimage of $\Hbar$ in $\Gtilde$.

In our notation above,
we can thus write $\Gtilde$ as the direct product
$\Gtilde = G_1 \times \cdots \times G_\ell$;
the subgroup $\Htilde$,
which is of maximal rank in $\Gtilde$,
is then of the form
$\Htilde = H_1 \times \cdots \times H_\ell$,
and each $H_i$
is a subgroup of maximal rank in $G_i$
(cf.~\cite[\S3]{BdS},
or the algebraic group version in~\cite[Prop.~4.1]{Le}).
Thus $\Gtilde/\Htilde$ is isomorphic to
$(G_1/H_1) \times \cdots \times (G_\ell/H_\ell)$.
The lemma follows.
\end{proof}

To proceed further,
let us first review
some preliminary results pertaining to
the rationality of $G/H$
in the greater generality
when $G$ is a connected reductive group
and $H \subseteq G$ is
any closed (connected) subgroup.

\begin{lemma}
\label{parabolic section}
Let $G$ be a connected reductive group,
and let $P \subseteq G$ be
a parabolic subgroup.
Then
$G/P$ is rational,
and the rational quotient map
$G \dasharrow G/P$
admits a rational section.
\end{lemma}

\begin{proof}
This is a standard consequence of
the Bruhat decomposition;
we give a detailed proof here
for the sake of clarity.

Let $B = T \cdot U \subseteq G$ be
a Borel subgroup
contained in $P$,
with maximal torus $T$
and unipotent radical $U$;
write $B^- = U^- \cdot T$
for the opposite Borel subgroup.
The parabolic subgroup $P$
is then the standard parabolic subgroup $P_I$
associated to a subset $I$
of the set of simple roots of $G$
relative to $T$.
By the Bruhat decomposition,
$G$ contains
the Zariski-dense open subset
(big cell)
$U^-\cdot T \cdot U$,
and
$P$ contains
the Zariski-dense open subset
$U_I^- \cdot T \cdot U$,
where $U_I^-$ denote
the subgroup of $U^-$ generated by
the 1-dimensional unipotent subgroups
associated to
the negative roots belonging to
the sub-root lattice spanned by $I$
(cf.~\cite[\S 30.1]{Hu}).
Write $V_I^-$ for the product variety of
the 1-dimensional unipotent subgroups
associated to
the negative roots in the complement of
the sub-root lattice spanned by $I$;
thus $V_I^-$ is a rational variety,
and one has
$U^- \cong V_I^- \times U_I^-$
as varieties.
Then the Bruhat decomposition above
shows that
the natural multiplication map
\[
  f
  \ :\
  V_I^- \times P
  \dasharrow
  G
  ,
  \qquad
  (v,p)
  \mapsto
  v \cdotp p
\]
is a birational map.
With respect to
the trivial action on $V_I^-$
and the natural right action
of $P$ on $P$ and on $G$,
the map $f$ is $P$-equivariant.
Thus
it induces the birational map
$V_I^- \dasharrow G/P$,
showing that $G/P$ is rational.
The inverse birational map
yields the desired rational section of
$G \dasharrow G/P$.
\end{proof}

\begin{corollary}
\label{subparabolic}
Let $G$ be a connected reductive group,
and let $H \subseteq G$ be any closed subgroup
contained in a parabolic subgroup $P$ of $G$.
Then
$G/H$ is birational to $(G/P) \times (P/H)$.
In particular,
if $P/H$ is rational,
then so is $G/H$.
\end{corollary}

\begin{proof}
The first assertion is deduced from lemma~\ref{spQ}
by the same argument as in
the proof of cor.~\ref{solv split},
using lemma~\ref{parabolic section} above
instead of
lemma~\ref{Rosenlict generic section}.
The second assertion then follows
by another use of lemma~\ref{parabolic section}.
\end{proof}

\begin{lemma}
[Borel-Tits]
\label{Borel-Tits}
Let $G$ be a connected reductive
(resp.~connected semisimple) group.
Suppose
$H \subseteq G$ is a connected closed subgroup
which is not contained in
any proper parabolic subgroup of $G$.
Then $H$ is reductive
(resp.~semisimple).
\end{lemma}

\begin{proof}
If $H$ is not reductive,
its unipotent radical
$U := R_u(H)$
is a non-trivial normal subgroup of $H$.
Since $G$ is connected reductive,
there exists
(cf.~\cite[\S30.3 Cor.~A]{Hu})
a parabolic subgroup $P$ of $G$
with $N_G(U) \subseteq P$
and $U \subseteq R_u(P)$.
The first inclusion gives
$H \subseteq P$;
the second inclusion forces
$P$ to be a proper parabolic subgroup,
contradicting our hypothesis on $H$.
Hence $H$ is reductive.

Now suppose $G$ is connected semisimple.
If $H$ is not semisimple,
its center $Z(H)$ is of positive dimension,
and we may choose
a non-trivial 1-parameter subgroup
$\lambda : \Gm \to G$
of $G$
with image in $Z(H)$.
By \cite[Def.~2.3/Prop.~2.6]{GIT},
there is a unique closed subgroup
$P(\lambda) \subseteq G$
characterized by the property that
\[
  \gamma \in P(\lambda)
  \ \iff\
  \lambda(t)
  \,
  \gamma
  \,
  \lambda(t^{-1})
  \quad
  \begin{array}[t]{l}
  \text{has a specialization in $G$}
  \\
  \text{when $t \in \Gm$ specializes to $0$}
  ;
  \end{array}
\]
moreover,
one knows that
$P(\lambda)$ is a parabolic subgroup of $G$,
and that
the image of $\lambda$
is contained in
the solvable radical of $P(\lambda)$.
As $G$ is semisimple,
this last fact
forces $P(\lambda)$ to be
a proper parabolic subgroup of $G$.
But since $H$ centralizes $\lambda$
by construction,
the characterizing property of $P(\lambda)$
shows that
$H \subseteq P(\lambda)$,
contradicting our hypothesis on $H$.
Thus $H$ is semisimple.
\end{proof}

Recall that
an algebraic group $M$ over $k$
is called \emph{special} (in the sense of Serre)
iff $H^1(K,M) = \{1\}$
for every field $K$ containing $k$.
By the classification theorem of
Serre and Grothendieck for special groups
(see for instance~\cite[Th.~5.4]{Re}),
one knows that
a connected semisimple group $M$ is special
if and only if
it is a direct product of
simply-connected almost-simple groups
of type~$A_n$ or~$C_n$
(i.e.~$\SL_{n+1}$ or $\Sp_{2n}$).
Special groups enjoy
the following important property:
if $X$ is an irreducible variety
on which
a special group $M$ acts
generically freely,
the rational quotient map
$X \dasharrow X/M$
admits a rational section
(cf.~\cite[Lemma~5.2 and Prop.~5.3]{Re}).

\begin{proposition}
\label{PropB}
Let 
$G$ be 
%
isomorphic to $\Sp_{2n}$
for some $n \geqslant 2$,
%
and
let $M \subsetneq G$ be
a maximal connected proper subgroup
which is semisimple and of maximal rank.
Then:
\begin{itemize}
\item[(a)]
$M$ is
a product of two simply-connected
almost-simple groups of type~$C$
(i.e.~$M$ is $G$-conjugate to
$\Sp_{2m} \times \Sp_{2n-m}$
for some $0 < m < n$),
and hence it is
a special group;
\item[(b)]
the rational quotient map
$G \dasharrow G/M$
admits a rational section;
\item[(c)]
for any closed subgroup $H \subseteq G$
contained in $M$,
$G/H$ is birational to $(G/M) \times (M/H)$;
\item[(d)]
if $\characteristic(k) \neq 2$,
then
$G/M$ is a rational variety.
\end{itemize}
\end{proposition}

\begin{proof}
By the Borel-de~Siebenthal algorithm
(cf.~\cite[Th.~6]{BdS}
and the algebraic group version
in~\cite[Prop.~6.6]{Le}),
one knows that
the Dynkin diagram of $M$
is obtained from
the extended Dynkin diagram of $G$
by removing
a vertex corresponding to
a simple root $\alpha$ of $G$
whose corresponding coefficient
$n_\alpha$ for the longest root $\alpha_0$ of $G$
is a prime number.
Moreover,
one has the exact sequence
\[
  1
  \rightarrow
  \mu(n'_\alpha)
  \rightarrow
  Z(\Mtilde)
  \rightarrow
  Z(M)
  \rightarrow
  1
  ,
\]
where $\Mtilde$ is
the simply-connected cover of $M$,
$n'_\alpha := \frac{|\alpha|^2}{|\alpha_0|^2} \cdot n_\alpha$,
and
$\mu(n'_\alpha)$ is the group of
roots of unity of order $n'_\alpha$.
Since $G$ is simply-connected
of type~$C$,
one has $n_\alpha = 2$
and $n'_\alpha = 1$,
and the Dynkin diagram of $M$
consists of two connected components
both of type~$C$.
Hence
$M = \Mtilde$ is simply-connected,
and is a product of two simply-connected
almost-simple groups of type~$C$.
Thus $M$ is a special group.
The right-multiplication action of $M$ on $G$
is generically free,
and hence by the property of
$M$ being a special group,
(cf.~\cite[Lemma~5.2 and Prop.~5.3]{Re}),
the rational quotient map
$G \dasharrow G/M$
has a rational section.
Hence by lemma~\ref{spQ},
for any closed subgroup $H$
contained in $M$,
$G/H$ is birational to $(G/M) \times (M/H)$.
This proves parts~(a), (b) and~(c).

For part~(d),
using \cite[Prop.~6.6]{Le} again,
we see that
$M$ is the centralizer of
an element of order $2$
(an involution) in $G$;
therefore,
$G/M$ is a symmetric variety
(which makes sense
since
$\characteristic(k) \neq 2$
by assumption).
It is well-known
(cf.~\cite[Th.~4.2, Cor.~4.3]{Sp})
that
a symmetric variety
is a spherical variety:
the natural left action on $G/M$
by a Borel subgroup $B$ of $G$
gives rise to
a Zariski-dense open orbit of $G/M$.
Consequently,
$G/M$ is birational to
a quotient variety of $B$,
and hence by cor.~\ref{Sol quotient},
it is a rational variety.
\end{proof}

\begin{setup}
\begin{proof}
[Proof of theorem~\ref{ThA}]
\label{ThAproof}
We can now establish our main theorem
in this section.
By lemma~\ref{reduction to adjoint factors},
we are reduced to showing that
when $G$ is a connected simply-connected
and almost-simple group of
type~$A$,
type~$C$ (when $\characteristic(k) \neq 2)$,
or
type~$B_3$ or $G_2$ (when $\characteristic(k) = 0)$,
and $H \subsetneq G$ is a connected proper closed subgroup
of maximal rank in $G$,
then $G/H$ is rational.
We proceed by induction on
the common rank~$n$
of $G$ and $H$,
the case of $n = 0$
being trivial.
Henceforth assume that $n \geqslant 1$,
and that our conclusion holds
for groups of the stated types
of lower ranks.

Suppose $H$ is contained in some
proper parabolic subgroup $P \subsetneq G$.
By cor.~\ref{subparabolic},
$G/H$ is birational to $(G/P) \times (P/H)$,
and by lemma~\ref{parabolic section},
$G/P$ is rational.
If $G$ is of type~$A$, $C$ or $G_2$
(resp.~type~$B_3$),
the adjoint factors of $P$
are all of type~$A$
(resp.~type~$A_1$ or $C_2$),
and the ranks of these factors
are strictly lower than that of $G$.
By lemma~\ref{reduction to adjoint factors}
applied to $H \subseteq P$
and our induction hypothesis,
we see that
$P/H$ is rational,
and hence
$G/H$ is rational.

If $G$ is of type~$A_n$
for $n \geqslant 1$,
the Borel-de~Siebenthal algorithm
shows that
every connected proper subgroup $H \subsetneq G$
of maximal rank in $G$
is contained in some proper parabolic subgroup of $G$;
our proof of theorem~\ref{ThA}
is therefore complete in this case.

If $G$ is of type~$C_n$
for $n \geqslant 2$
and $\characteristic(k) \neq 2$,
we are reduced to the case when
$H \subsetneq G$ is not contained in
any proper parabolic subgroup of $G$
and is therefore semisimple
by lemma~\ref{Borel-Tits}.
We let $M \subsetneq G$ be
a maximal connected proper subgroup
containing $H$;
thus $M$ is also of maximal rank in $G$,
and is not contained in
any proper parabolic subgroup of $G$,
and by lemma~\ref{Borel-Tits},
$M$ is semisimple.
By prop.~\ref{PropB},
$G/H$ is birational to $(G/M) \times (M/H)$,
and
$G/M$ is rational
(because $\characteristic(k) \neq 2$);
moreover,
the adjoint factors of $M$
are all of type~$C$,
and the ranks of these factors
are strictly lower than that of $G$.
By lemma~\ref{reduction to adjoint factors}
applied to $H \subseteq M$
and our induction hypothesis,
we see that
$M/H$ is rational,
and hence
$G/H$ is rational.

In the remaining cases,
$G$ is of type~$B_3$ or $G_2$
with $\characteristic(k) = 0$,
and
$H \subsetneq G$ is not contained in
any proper parabolic subgroup of $G$;
again,
$H$ is semisimple
by lemma~\ref{Borel-Tits}.
The rationality of $G/H$
is then established directly
using results in section~\ref{low dim homog ctd},
and we defer the proof of these cases
to cor.~\ref{B23C3G2}.
The proof of theorem~\ref{ThA}
is thus completed
---
modulo the use of cor.~\ref{B23C3G2}
for the low dimensional cases.
\end{proof}
\end{setup}

\begin{remark}
It is possible that
the assumptions in theorem~\ref{ThA}
on the adjoint factors of $G$
can be removed altogether.
This would be the case
if the assertion of prop.~\ref{PropB}(b)
can be established
for almost-simple groups of any type;
our induction argument in~\ref{ThAproof}
would then yield
the stable-rationality of $G/H$ in general.
In turn,
the rationality of $G/H$ in general
would be reduced to
the assertions of prop.~\ref{PropB}(d)
for almost-simple groups of any type;
i.e.~to the rationality of $G/M$
when $G$ is almost-simple (of any type)
and $M \subsetneq G$ is
a maximal connected proper subgroup
of maximal rank in $G$
but which is not contained in
any proper parabolic subgroup of $G$.
As explained in~\ref{ThAproof},
such an $M$ is semisimple,
and its (finitely many) possibilities
are determined by
the Borel-de~Siebenthal algorithm.
\end{remark}

%
%
%
%


\section{Low dimensional homogeneous varieties}
\label{low dim homog}

From now on,
we work over
an algebraically closed base field $k$
of characteristic~0.
In this and the next section,
we apply geometric methods
to study the rationality problem~\ref{RatPbm}.
Our goal is to establish theorems~\ref{ThB0} and~\ref{ThB}
asserting the rationality of
all homogeneous varieties $G/H$
of sufficiently low dimensions,
thereby answering
the rationality problem~\ref{RatPbm} affirmatively
in these cases.
In this section,
we place no restriction on
the connectedness of $H$,
while in section~\ref{low dim homog ctd},
we extend our rationality results further
when $H$ is assumed to be connected.
The following argument
will be used several times
in both sections.

\begin{lemma}
\label{finbir}
Let $H_1 \subseteq H_2$ be
two connected algebraic groups such that $H_2$ (and hence $H_1$)
act rationally on an algebraic variety $X$,
and let $f : X/H_1 \dasharrow X/H_2$ be
the dominant rational quotient map.
The following are equivalent:
\begin{itemize}
\item[(a)]
$f$ is birational;
\item[(b)]
$f$ is generically injective;
\item[(c)]
$f$ is generically finite;
\item[(d)]
$\dim(X/H_1) = \dim(X/H_2)$;
\item[(e)]
for all points $x \in X$
in general position,
its orbits $x \cdotp H_1$ and $x \cdotp H_2$
under the action of $H_1$ and $H_2$
have the same Zariski closure.
\end{itemize}
\end{lemma}

\begin{proof}
The implications
$\text{(a)}
 \Rightarrow
 \text{(b)}
 \Rightarrow
 \text{(c)}
 \Rightarrow
 \text{(d)}
$
are clear.
To show the other implications,
we replace $X$ by another birational model
and assume that
both $H_1 \subseteq H_2$ act regularly on $X$
(cf.~\cite[Th.~1]{Ro56}).
For a point $x \in X$ in general position,
its orbit $x \cdotp H_i$
under the action of $H_i$
is an irreducible locally closed subvariety in $X$
of dimension
$\dim(x \cdotp H_i) = \dim X - \dim(X/H_i)$.
If $\dim(X/H_1) = \dim(X/H_2)$,
these orbits are of the same dimension,
and since we have the inclusion
$x \cdotp H_1 \subseteq x \cdotp H_2$,
these orbits have the same Zariski closure in $X$;
hence
$\text{(d)} \Rightarrow \text{(e)}$.
The fiber of $f$ over
the point $x \cdotp H_2$ in $X / H_2$
consists of
those points $x' \cdotp H_1$ in $X / H_1$
which, when regarded as orbits in $X$,
belong to the same Zariski closure
as $x \cdotp H_2$ in $X$;
hence
$\text{(e)} \Rightarrow \text{(b)}$.
Finally, recall that
(cf.~\cite[\S 4.6, Th.]{Hu})
a dominant injective morphism between
irreducible varieties
induces a finite purely inseparable extension
of their function fields.
As our base field $k$
is of characteristic~0,
we infer that
when $f$ is generically injective,
it induces a trivial extension of function fields
and $f$ is therefore birational;
hence $\text{(b)} \Rightarrow \text{(a)}$.
\end{proof}

Recall that
a unirational curve over any field
is rational
by L\"uroth's theorem,
and a unirational surface over
an algebraically closed field of characteristic~0
is rational
by Castelnuovo's rationality criterion.
Hence
over our algebraically closed base field $k$
of characteristic~0,
a unirational variety is rational
if its dimension is~$\leqslant 2$.

\begin{lemma}
\label{dim1or2}
Let $G$ be a connected linear algebraic group,
and let $B \subseteq G$ be
a Borel subgroup of $G$.
For any closed subgroup $H \subseteq G$,
if $\dim(B \backslash G/H) \leqslant 2$,
then $G/H$ is rational.
\end{lemma}

\begin{proof}
The underlying variety of $G$
is rational
(cf.~\cite{Ch});
the space of double cosets
$B \backslash G/H$,
being dominated by $G$,
is therefore
a unirational variety.
Hence
our hypothesis on its dimension
implies that
$B \backslash G/H$ is rational.
By lemma~\ref{Sol}
applied to the left action of $B$ on $G/H$,
it now follows that
$G/H$ is rational.
\end{proof}

\begin{lemma}
\label{TB3}
Let $G$ be a connected semisimple group
with maximal torus $T$,
and let $B \subseteq G$ be
a Borel subgroup of $G$
containing $T$.
For any closed subgroup $H \subseteq G$,
one has
$\dim(B \backslash G/H) \leqslant \dim(T \backslash G/H)$;
and if equality holds,
then $G/H$ is rational.
\end{lemma}

\begin{proof}
Let $U$ be the unipotent radical of $B$.
The inclusion of $T$ in $B = T \ltimes U$
induces a dominant rational map
$T \backslash G/H \dasharrow B \backslash G/H$,
so one always has
$\dim(B \backslash G/H) \leqslant \dim(T \backslash G/H)$.
Assume that equality holds.
By lemma~\ref{finbir}
applied to the left action of
$T$ and $B$ on $G/H$,
we see that
for a point $x \in G/H$
in general position,
its orbits
$T \cdotp x$ and $B \cdotp x$
under the action of $T$ and $B$
have the same Zariski closure in $G/H$.
Since $B = U \cdot T$
%
%
%
and since $G$ contains
the Zariski-dense open subset
(big cell)
$U^-\cdotp T \cdotp U$ (cf.~\cite[\S 28.5]{Hu}),
%
%
%
$$G/H = G \cdotp x =
\overline{U^- B \cdotp x} =
\overline{U^- T \cdotp x} =
\overline{B^- \cdotp x} .$$
Hence
$G/H$ is birational to $B^-/B^-_{x}$
where $B^-_{x}$ the stabilizer of $x$ in $B^-$;
and since $B^-$ is connected solvable,
it follows from cor.~\ref{Sol quotient} that
this is rational.
\end{proof}

\begin{theorem}
\label{ThB0}
Let $G$ be a connected linear algebraic group,
and let $H \subseteq G$ be
any closed subgroup.
If $\dim(G/H) \leqslant 5$,
then $G/H$ is a rational variety.
\end{theorem}

\begin{proof}
By lemma~\ref{mod radical},
we may replace $G$
by its maximal semisimple quotient $G'$
and $H$ by its image $H'$ in $G'$;
the rationality of $G'/H'$
implies that of $G/H$,
but the dimension of $G'/H'$
can only be at most that of $G/H$.
Henceforth,
we assume $G$ is semisimple.

Let $T$ be a maximal torus of $G$,
and let $B$ be a Borel subgroup of $G$
containing $T$.
If $\dim T \leqslant 1$,
the semisimple group $G$
is either trivial or of rank~1,
isogenous to $\SL_2$;
in either case,
we see that
\[
  \dim(B \backslash G/H)
  \ \leqslant\
  \dim(B \backslash G)
  \ \leqslant\
  1
  ,
\]
and the rationality of $G/H$
follows from lemma~\ref{dim1or2}.
Henceforth we assume
$\dim T \geqslant 2$.

Replacing $G$ further by
its image in $\Aut(G/H)$
if necessary,
we may also assume that
the natural left action of $G$ on $G/H$
is generically faithful.
Applying lemma~\ref{Demazure generic freeness}
to the left action of $T$ on $G/H$,
we have
$\dim(G/H)
 =
 \dim(T \backslash G/H) + \dim T
$.
This is~$\leqslant 5$
by hypothesis,
so $\dim(T \backslash G/H) \leqslant 3$.
But as one always has
$\dim(B \backslash G/H) \leqslant \dim(T \backslash G/H)$,
this means that
\[
  \text{either}
  \qquad
  \dim(B \backslash G/H) \leqslant 2
  \qquad
  \text{or}
  \qquad
  \dim(B \backslash G/H) = \dim(T \backslash G/H) = 3
  ,
\]
and the rationality of $G/H$
follows from lemma~\ref{dim1or2} and lemma~\ref{TB3}
respectively.
\end{proof}
%

%
%
%
%


\section{Low dimensional homogeneous varieties, continued}
\label{low dim homog ctd}

In this section,
we consider the rationality of $G/H$
when $H \subseteq G$ is
a \emph{connected} closed subgroup.
We still work over
an algebraically closed base field $k$
of characteristic~0.
In the series of lemmas below
leading up to the main theorem~\ref{ThB}
of this section,
we adopt the following
hypotheses and notation.

\begin{setup}
\label{notA}
\quad
Let $G$ be a connected semisimple group,
and let $H \subseteq G$ be
a connected closed subgroup.
We fix once and for all:
\begin{center}
\ \\[-3ex]
\begin{tabular}{l l}
  $U(H)$
  &
  the unipotent radical of $H$,
  \\
  $S$
  &
  a (reductive) Levi subgroup of $H$,
  so that $H = S \ltimes U(H)$;
  \\
  $T_H$
  &
  a maximal torus of $S$ (and hence of $H$),
  \\
  $B_S^{\pm}$
  &
  a pair of opposite Borel subgroups of $S$
  containing $T_H$,
  \\
  $U_S^{\pm}$
  &
  the unipotent radical of the corresponding $B_S^{\pm}$,
  so that $B_S^{\pm} = T_H \ltimes U_S^{\pm}$;
  \\
  $B_H^{\pm}$
  &
  the preimage in $H$ of $B_S^{\pm}$,
  so that $B_H^{\pm} = B_S^{\pm} \ltimes U(H)$,
  \\
  $U_H^{\pm}$
  &
  the preimage in $H$ of $U_S^{\pm}$
  so that $U_H^{\pm} = U_S^{\pm} \ltimes U(H)$.
\end{tabular}
\ \\[1ex]
\end{center}
Here,
$B_H^{\pm}$ are Borel subgroups of $H$,
with unipotent radicals $U_H^{\pm}$,
and $T_H$ is a maximal torus of $H$
contained in $B_H$.
Having fixed these,
we choose:
\begin{center}
\ \\[-3ex]
\begin{tabular}{l l}
  $B = B^+$
  &
  a Borel subgroup of $G$
  containing $B_H$;
  \\
  $T$
  &
  a maximal torus of $B$ (and hence of $G$)
  containing $T_H$;
  \\
  $B^-$
  &
  the opposite Borel subgroup of $G$
  containing $T$,
  such that $B^- \cap B = T$;
  \\
  $U^\pm$
  &
  the unipotent radical of
  the corresponding $B^\pm$,
  so that $B^\pm = T \ltimes U^\pm$.
\end{tabular}
\ \\[1ex]
\end{center}
We also set
\[
  \begin{aligned}
    u(H)
    &
    \ :=\
    \dim U(H)
    ,
    \\
    u_G
    &
    \ :=\
    \dim U = \dim U^-
    ,
    &
    \quad
    t_G
    &
    \ :=\
    \dim T
    \quad
    \text{(the rank of $G$)}
    ,
    \\
    u_H
    &
    \ :=\
    \dim U_H = \dim U_H^-
    ,
    &
    \quad
    t_H
    &
    \ :=\
    \dim T_H
    \quad
    \text{(the rank of $H$)}
    .
  \end{aligned}
\]

\noindent
Thus:
\[
  \begin{aligned}
    \dim S
    &
    \ =\
    \dim T_H + 2\dim U_S
    &
    =\
    &
    t_H + 2(u_H - u(H))
    ,
    \\
    \dim H
    &
    \ =\
    \dim S + \dim U(H)
    &
    =\
    &
    t_H + 2 u_H - u(H)
    ,
    \\
    \dim G
    &
    \ =\
    \dim T + 2\dim U
    &
    =\
    &
    t_G + 2u_G
    ,
  \end{aligned}
\]
and hence
\[
  \dim G/H
  \ =\
  \dim G - \dim H
  \ =\
  (t_G - t_H) + 2(u_G - u_H) + u(H)
  .
  \tag{$*$}
  \label{dimformula}
\]
The subgroup inclusion maps
$U_H \subseteq B_H$,
$B_H \subseteq H$
and
$U^- \subseteq B^-$
induce the dominant rational maps
$\alpha$, $\gamma$ and $\varphi$
between the respective spaces of
double cosets:
\[
  \begin{matrix}
  B^- \backslash G / U_H
  &
  \ \overset{\textstyle\alpha}{\dashrightarrow}\
  &
  B^- \backslash G / B_H
  &
  \ \overset{\textstyle\varphi}{\dashleftarrow}\
  &
  U^- \backslash G / B_H
  \\[-1ex]
  &
  &
  \shortmid
  &
  &
  \\[-1ex]
  &
  &
  \downarrow^{\rlap{$\ \gamma$}}
  &
  &
  \\
  &
  &
  B^- \backslash G / H
  &
  &
  \\
  \end{matrix}
  \tag{$**$}
  \label{dblcosetsdiagram}
\]
We will consider
these rational maps
in the series of lemmas below
leading up to theorem~\ref{ThB}.
\end{setup}

\begin{lemma}
\label{uGH3}
In the situation of~\ref{notA}:
\begin{itemize}
\item[(a)]
$B^- \backslash G/U_H$
is rational,
of dimension~$u_G - u_H$.
\item[(b)]
One has
$\dim(B^- \backslash G/H)
 \leqslant
 \dim(B^- \backslash G/U_H)
$;
and if equality holds,
then $G/H$ is rational.
\item[(c)]
If $u_G - u_H \leqslant 3$,
then $G/H$ is rational.
\end{itemize}
\end{lemma}

\begin{proof}
The Bruhat (big cell) decomposition of $G$
shows that
$B^- \backslash G / U_H$ contains
a Zariski-dense constructible subset
$B^- \backslash  B^- U / U_H$
which is birational to $U/U_H$;
in turn,
this is rational
by cor.~\ref{Sol quotient}.
Since $\dim(U/U_H) = u_G - u_H$,
we see that
$B^- \backslash G / U_H$ is rational
and of that dimension;
this shows part~(a).

The asserted inequality of part~(b)
follows from the existence of
the dominant rational map
$\gamma \circ \alpha$
in the diagram~\eqref{dblcosetsdiagram} of~\ref{notA}.
If equality holds,
lemma~\ref{finbir}
applied to the right action of
$U_H$ and $H$ on $B^- \backslash G$
shows that
$\gamma \circ \alpha$ is a birational map,
and hence
$B^- \backslash G / H$ is rational;
by lemma~\ref{Sol}
applied to the left action of $B^-$ on $G/H$,
it then follows that
$G/H$ is rational.

For part~(c),
if $\dim(B^- \backslash G/H) \leqslant 2$,
the rationality of $G/H$
follows from lemma~\ref{dim1or2};
henceforth,
assume that
$\dim(B^- \backslash G / H) \geqslant 3$.
Our hypothesis
together with parts~(a) and~(b)
then yield
\[
  3
  \ \leqslant\
  \dim(B^- \backslash G / H)
  \ \leqslant\
  \dim(B^- \backslash G / U_H)
  \ \leqslant\
  3
  ,
\]
whence equality holds throughout,
and the rationality of $G/H$
follows from part~(b) again.
\end{proof}

\begin{remark}
\label{rmkBGBHrational}
Although we do not need it below,
it is of interest to note that
$B^- \backslash G / B_H$
is in fact rational.
Indeed,
$B^- \backslash G /U_H$
contains the Zariski-open subset
$B^- \backslash B^- U / U_H \cong U/U_H$,
and with respect to
the (regular) action of $T_H$ on $U/U_H$ by conjugation,
the isomorphism
is $T_H$-equivariant.
Thus
the quotient
$B^- \backslash G /B_H$
of
$B^- \backslash G /U_H$ by $T_H$
is birational to
the quotient of $U/U_H$ by $T_H$.
Moreover,
$U/U_H$ is $T_H$-equivariantly isomorphic
to the quotient
$V := \Lie(U) / \Lie(U_H)$
of Lie algebras,
on which $T_H$ acts linearly
via its adjoint actions on
$\Lie(U)$ and $\Lie(U_H)$.
So $B^- \backslash G /B_H$
is birational to
the quotient $V/T_H$.
Now choose a basis of $V$
for which the action of $T_H$
is diagonal,
and let
$V_0 \subseteq V$
denote the open subset
on which all coordinates are nonzero.
Then $V_0$ is isomorphic to
a torus,
on which $T_H$ acts by multiplication;
thus $V_0/T_H$ is a torus as well.
This shows that
$V_0/T_H$
and hence $V/T_H$ and $B^- \backslash G /B_H$
are all rational.
\end{remark}

\begin{lemma}
\label{orbit}
In the situation of~\ref{notA},
one has
$\dim(B^- \backslash G / H)
 \leqslant
 \dim(B^- \backslash G / B_H)
$;
and if equality holds,
then $G/H$ is rational.
\end{lemma}

\begin{proof}
The asserted inequality
follows from the existence of
the dominant rational map
$\gamma$
in the diagram~\eqref{dblcosetsdiagram} of~\ref{notA}.
If equality holds,
then by lemma~\ref{finbir}
applied to the right action of
$B_H$ and $H$
on $B^- \backslash G$,
we see that
$\gamma
 :
 B^- \backslash G / B_H
 \dasharrow
 B^- \backslash G / H
$
is birational.
By lemma~\ref{uGH3},
the variety
$X_1 := B^- \backslash G /U_H$ is rational,
of dimension $u_G - u_H$.
The torus $M_1 := B_H/U_H$
acts by right multiplication on $X_1$
with quotient
$X_1/M_1 = B^- \backslash G / B_H$.
On the other hand,
the variety
$X_2 := G/H$ is of dimension~$(t_G - t_H) + 2(u_G - u_H) + u(H)$
by the formula~\eqref{dimformula} in~\ref{notA}.
The solvable group $M_2 := B^-$
acts by left multiplication on $X_2$
with quotient
$M_2 \backslash X_2 = B^- \backslash G / H$.
Thus
$X_1/M_1$ is birational to $M_2 \backslash X_2$
via $\gamma$;
since $\dim X_1 \leqslant \dim X_2$,
lemma~\ref{no-name-lemma}
is applicable
and shows that
the rationality of $X_1$
implies that of $X_2 = G/H$.
This completes
the proof of the lemma.
\end{proof}

\begin{lemma}
\label{BU}
In the situation of~\ref{notA},
set
\[
  d
  \ :=\
  \dim(U^- \backslash G / B_H)
  -
  \dim (B^- \backslash G / B_H)
  .
\]
\begin{itemize}
\item[(a)]
Let $L$ denote
the identity component of
the kernel $L'$ of
the natural left action of $T$ on
$U^- \backslash G/B_H$.
Then $d = t_G - t_L$,
where $t_L := \dim L$.
\item[(b)]
If $d \leqslant 1$,
then $G/H$ is rational.
\end{itemize}
\end{lemma}

\begin{proof}
First note that
$d \geqslant 0$
by the existence of
the dominant rational map $\varphi$
in the diagram~\eqref{dblcosetsdiagram} of~\ref{notA}.
The torus $L$
is of finite index in
the diagonalizable group $L'$
contained in $T$;
hence
$\dim(T/L') = \dim(T/L) = t_G - t_L$.
By construction,
the induced left action of $T/L'$
on $U^- \backslash G/B_H$
is generically faithful,
and its quotient is
$TU^- \backslash G/B_H
 =
 B^- \backslash G/B_H
$.
Hence
by lemma~\ref{Demazure generic freeness},
we have
\[
  \dim(U^- \backslash G/B_H)
  \ =\
  \dim(B^- \backslash G/B_H)
  +
  \dim(T/L')
  ,
\]
from which it follows that
$d = \dim(T/L') = t_G - t_L$.
This shows part~(a) of the lemma.

Let $D := C_G(L)$ denote
the centralizer of $L$ in $G$;
it is a connected reductive subgroup of $G$
(cf.~\cite[\S 26.2, Cor.~A; \S 22.3, Th.]{Hu})
containing the maximal torus $T$,
and a Borel subgroup
is given by $B \cap D$
(cf.~\cite[\S 22.4, Cor.]{Hu}),
whose unipotent radical is
$U_D := U \cap D$.
We set
$u_D := \dim U_D$.

We claim that
$U$ is contained in
the image $U_D \cdotp U_H$
of multiplying
$U_D$ and $U_H$ in $G$.
Assuming this for the moment,
we infer that
$u_G \leqslant u_D + u_H$,
and part~(b) of the lemma
can be deduced from this
as follows.
If $d = 0$,
then $t_L = t_G$,
so $L = T$ is the maximal torus of $G$,
and it is self-centralizing in $G$
(cf.~\cite[\S 26.2, Cor.~A]{Hu}),
whence $D = T$,
and we have
$u_D = 0$.
If $d = 1$,
then $L$ is a subtorus
of codimension~1 in $T$;
if $L$ is a regular subtorus,
then $D = T$
and we have
$u_D = 0$
as before;
if $L$ is a singular subtorus,
then $D$ is isogenous to $L \times \SL_2$
(cf.~\cite[\S 26.2, Cor.~B]{Hu}),
and we have
$u_D = 1$.
In any case,
we see that
$d \leqslant 1$
implies
$u_G - u_H \leqslant u_D \leqslant 1 \leqslant 3$,
and the rationality of $G/H$
follows from lemma~\ref{uGH3}.

We now proceed to prove our claim that
$U \subseteq U_D \cdotp U_H$.
The torus $T$ normalizes $U^-$,
and so it acts regularly from the left
on $U^- \backslash G/B_H$,
which contains
$U^- \backslash U^- B/B_H$
(isomorphic to $B/B_H$)
as a Zariski-dense $T$-stable open subset,
by the Bruhat (big cell) decomposition of $G$.
Hence
$L$ acts trivially from the left on
$U^- \backslash U^- B/B_H$.
This means that
for any $b \in B$ and any $\ell \in L$,
one has
$U^- \cdotp \ell \cdotp b \cdotp B_H
 =
 U^- \cdotp b \cdotp B_H
$,
whence
\[
  \ell \cdotp b
  \ =\
  v_1 \cdotp b \cdotp b_1
  \qquad
  \text{for some}
  \quad
  v_1 \in U^-
  ,
  \
  b_1 \in B_H
  .
\]
Thus $v_1 \in U^- \cap B = \{1\}$,
and if we write
$b_1 = t_1 \cdotp u_1$
(with $t_1 \in T_H$, $u_1 \in U_H$)
and
$b = t \cdotp u$
(with $t \in T$, $u \in U$),
then reducing modulo~$U$ shows that
$\ell = t_1$ in $T$.
Hence,
for any $b \in B$
and $\ell \in L$,
there exists $u_1 \in U_H$
such that
\[
  b^{-1} \cdotp \ell \cdotp b
  \ =\
  \ell \cdotp u_1
  .
\]
Specializing this relation
to the case when $b$ lies in $U_H$,
we see that
$L$ normalizes $U_H$,
and hence
$L \cdotp U_H$ is a connected subgroup of $B$
containing $L$ as a maximal torus.
Specializing the relation
to the case when $b$ equals $u \in U$,
we see that
$u^{-1} \cdotp L \cdotp u$
is also a maximal torus
in $L \cdotp U_H$,
so it is $U_H$-conjugate to $L$:
there exists $u_2 \in U_H$
such that
$(u \cdotp u_2)^{-1} \cdotp L \cdotp (u \cdotp u_2)
 =
 L
$,
or equivalently,
$u \cdotp u_2 \in U \cap N_G(L)$.
But since
$U$ is connected and solvable,
by~\cite[\S 19.4, Prop.]{Hu}),
$U \cap N_G(L) = N_U(L)$
is equal to
$C_U(L) = U \cap C_G(L) = U \cap D = U_D$.
We have thus shown that
for any $u \in U$,
there exists $u_2 \in U_H$
(depending on $u$)
such that
$u \cdotp u_2 \in U_D$.
Hence
$U$ is contained in $U_D \cdotp U_H$,
and our claim follows.
\end{proof}

\begin{lemma}
\label{tu6}
In the situation of~\ref{notA}:
\begin{itemize}
\item[(a)]
$U^- \backslash G/B_H$
is rational,
of dimension~$(t_G - t_H) + (u_G - u_H)$.
\item[(b)]
If $(t_G - t_H) + (u_G - u_H) \leqslant 5$,
then $G/H$ is rational.
\end{itemize}
\end{lemma}

\begin{proof}
The Bruhat (big cell) decomposition of $G$
shows that
$U^- \backslash G/B_H$ contains
a Zariski-dense constructible subset
$U^- \backslash  U^- B / B_H$
which is birational to $B/B_H$;
in turn,
this is rational
by cor.~\ref{Sol quotient}.
Since $\dim(B/B_H) = (t_G - t_H) + (u_G - u_H)$,
we see that
$U^- \backslash G / B_H$ is rational
and of that dimension;
this shows part~(a).

For part~(b),
consider the dominant rational maps
$\gamma$ and $\varphi$
in the diagram~\eqref{dblcosetsdiagram} of~\ref{notA},
which give the inequalities
\[
  \dim(B^- \backslash G / H)
  \ \leqslant\
  \dim(B^- \backslash G / B_H)
  \ \leqslant\
  \dim(U^- \backslash G / B_H)
  .
\]
If $\dim(B^- \backslash G/H) \leqslant 2$,
the rationality of $G/H$
follows from lemma~\ref{dim1or2},
while if one has
$\dim(B^- \backslash G / H)
 =
 \dim(B^- \backslash G / B_H)
$,
the rationality of $G/H$
follows from lemma~\ref{orbit}.
In the remaining cases,
our hypothesis yields
\[
  4
  \ \leqslant\
  \dim(B^- \backslash G / H) + 1
  \ \leqslant\
  \dim(B^- \backslash G / B_H)
  \ \leqslant\
  \dim(U^- \backslash G / B_H)
  \ \leqslant\
  5
  ,
\]
whence
$d := \dim(U^- \backslash G / B_H) - \dim(B^- \backslash G / B_H)$
is~$\leqslant 1$;
the rationality of $G/H$
now follows from lemma~\ref{BU}.
\end{proof}

\begin{lemma}
\label{Hbu}
In the situation of~\ref{notA},
suppose $H$ is reductive;
set
\[
  e
  \ :=\
  \dim (B^- \backslash G / U_H)
  -
  \dim (B^- \backslash G / B_H)
  .
\]
\begin{itemize}
\item[(a)]
Let $K$ denote
the identity component of
the kernel $K'$ of
the natural right action of $T_H$ on
$B^- \backslash G/U_H$.
Then $e = t_H - t_K$,
where $t_K := \dim K$.
\item[(b)]
If the natural left action of $G$ on $G/H$
has zero-dimensional kernel,
then $t_K = 0$
and hence $e = t_H$.
\item[(c)]
If $e = 0$,
which is to say
$\dim (B^- \backslash G / U_H)
 =
 \dim (B^- \backslash G / B_H)
$,
then $G/H$ is rational.
\end{itemize}
\end{lemma}

\begin{proof}
Part~(a) of the lemma
is established
along the same lines as
part~(a) of lemma~\ref{BU},
using the generically faithful
right action of $T_H/K'$
on $B^- \backslash G / U_H$,
passing to the quotient
and applying lemma~\ref{Demazure generic freeness}
to get
\[
  \dim(B^- \backslash G/U_H)
  \ =\
  \dim(B^- \backslash G/B_H)
  +
  \dim(T_H/K')
  .
\]
%


Let $E := C_G(K)$ denote
the centralizer of $K$ in $G$;
it is a connected reductive subgroup of $G$
(cf.~\cite[\S 26.2, Cor.~A; \S 22.3, Th.]{Hu})
containing the maximal torus $T$,
and a Borel subgroup
is given by $B \cap E$
(cf.~\cite[\S 22.4, Cor.]{Hu}),
whose unipotent radical is
$U_E := U \cap E$.
Let $E_H := C_H(K) = E \cap H$ denote
the centralizer of $K$ in $H$;
it is a connected reductive subgroup of $H$
containing the maximal torus $T_H$,
and a Borel subgroup
is given by $B_H \cap E_H = B \cap E \cap H$,
whose unipotent radical is
$U_{E_H} := U_H \cap E_H = U \cap E \cap H$.
We set
$u_E := \dim U_E$
and
$u_{E_H} := \dim U_{E_H}$;
hence
$\dim E = t_G + 2u_E$
and
$\dim E_H = t_H + 2u_{E_H}$.

We claim that
$U$ is equal to
the Zariski closure
$\overline{U_E \cdotp U_H}$
of the image of multiplying
$U_E$ and $U_H$ in $G$.
Assuming this for the moment,
we infer that
$u_G \leqslant u_E + u_H$.
More precisely,
since the multiplication map
$U_E \times U_H \to \overline{U_E \cdotp U_H} = U$
has a general fiber isomorphic to
$U_{E_H} = U_E \cap U_H$,
we see that
$u_G + u_{E_H} = u_E + u_H$.
We have the natural inclusion map
\[
  E/E_H
  \ =\
  E/(E \cap H)
  \ \hookrightarrow\
  G/H
  .
\]
Here,
since $H$ is reductive by hypothesis,
one has $u(H) = 0$,
and so
by the dimension formula~\eqref{dimformula} in~\ref{notA},
we have
\[
  \dim(G/H)
  \ =\
  (t_G - t_H) + 2(u_G - u_H)
  .
\]
On the other hand,
\[
  \dim(E/E_H)
  \ =\
  (t_G + 2u_E) - (t_H + 2u_{E_H})
  \ =\
  (t_G - t_H) + 2(u_G - u_H)
  .
\]
Thus $\dim E/E_H = \dim G/H$,
which shows that
the locally closed subvariety $E/E_H$ of $G/H$
is in fact a Zariski-open subset,
whence
$G/H$ and $E/E_H$ are birational to each other.
This is the key fact needed
for showing parts~(b) and~(c) of the lemma.
For part~(b),
we note that
$K$ is a normal subgroup of both $E$ and $E_H$;
if we let
$\overline{E} := E/K$ and $\overline{E_H} := E_H/K$
denote the respective quotient groups,
then $E/E_H$ is
naturally isomorphic to
$\overline{E}/\overline{E_H}$,
and the morphisms
\[
  \overline{E}/\overline{E_H}
  \ \cong\
  E/E_H
  \ \hookrightarrow\
  G/H
\]
are $E$-equivariant
with respect to
the natural left actions of $E$.
But $K$ acts trivially on
$\overline{E}/\overline{E_H}$
by construction,
so it also acts trivially on $G/H$.
Hence $K$ is contained in
the kernel of the natural left action
of $G$ on $G/H$;
if this kernel is zero-dimensional,
so is $K$,
which is to say
$t_K = 0$.
For part~(c),
if we have $e = 0$,
then $t_K = t_H$,
so $K = T_H$ is the maximal torus of $H$,
and it is therefore
self-centralizing in $H$
(cf.~\cite[\S 26.2, Cor.~A]{Hu}),
whence $E_H = E \cap H = C_H(K)$
is equal to $T_H$.
Then $E/E_H = E/T_H$ is rational
by theorem~\ref{Th0},
and the rationality of $G/H$
follows.

To prove our claim,
it suffices to show that
a point in general position in $U$
belongs to (the Zariski closure of)
$U_E \cdotp U_H$,
since the reverse inclusion is clear.
The torus $T_H$ normalizes $U_H$,
and so it acts regularly from the right
on $B^- \backslash G/U_H$,
which contains
$B^- \backslash B^- U/U_H$
(isomorphic to $U/U_H$)
as a Zariski-dense open subset,
by the Bruhat (big cell) decomposition of $G$.
Hence
$K$ acts trivially from the right on
$B^- \backslash B^- U/U_H$.
This means that
for a point $u \in U$
in general position,
and for any $k \in K$,
one has
$B^- \cdotp u \cdotp k \cdotp U_H
 =
 B^- \cdotp u \cdotp U_H
$,
whence
\[
  u \cdotp k
  \ =\
  v_1 \cdotp t_1 \cdotp u \cdotp u_1
  \qquad
  \text{for some}
  \quad
  v_1 \in U^-
  ,
  \
  t_1 \in T
  ,
  \
  u_1 \in U_H
  .
\]
Thus $v_1 \in U^- \cap B = \{1\}$,
and reducing modulo~$U$ shows that
$k = t_1$ in $T$;
hence
$u^{-1} \cdotp k \cdotp u
 =
 k \cdotp u_1^{-1}
$
lies in $K \cdotp U_H$.
Note that
$K \cdotp U_H$ is a connected subgroup of $B_H$
containing $K$ as a maximal torus;
the above discussion shows that
$u^{-1} \cdotp K \cdotp u$
is also a maximal torus
in $K \cdotp U_H$,
so it is $U_H$-conjugate to $K$:
there exists $u_2 \in U_H$
such that
$(u \cdotp u_2)^{-1} \cdotp K \cdotp (u \cdotp u_2)
 =
 K
$,
or equivalently,
$u \cdotp u_2 \in U \cap N_G(K)$.
But since
$U$ is connected and solvable,
by~\cite[\S 19.4, Prop.]{Hu}),
$U \cap N_G(K) = N_U(K)$
is equal to
$C_U(K) = U \cap C_G(K) = U \cap E = U_E$.
We have thus shown that
for a point $u \in U$ in general position,
there exists $u_2 \in U_H$
(depending on $u$)
such that
$u \cdotp u_2 \in U_E$.
This establishes our claim
and hence completes the proof of the lemma
as well.
\end{proof}

\begin{lemma}
\label{uGH4}
In the situation of~\ref{notA},
suppose $H$ is reductive;
if $u_G - u_H \leqslant 4$,
then $G/H$ is rational.
\end{lemma}

\begin{proof}
Consider the dominant rational maps
$\gamma$ and $\alpha$
in the diagram~\eqref{dblcosetsdiagram} of~\ref{notA},
which give the inequalities
\[
  \dim(B^- \backslash G / H)
  \ \leqslant\
  \dim(B^- \backslash G / B_H)
  \ \leqslant\
  \dim(B^- \backslash G / U_H)
  .
\]
If $\dim(B^- \backslash G / H) \leqslant 2$,
the rationality of $G/H$
follows from lemma~\ref{dim1or2};
henceforth,
assume that
$\dim(B^- \backslash G / H) \geqslant 3$.
We have
$\dim(B^- \backslash G / U_H) \leqslant 4$
by our hypothesis and lemma~\ref{uGH3}.
Hence among the two inequalities
in the above display,
equality holds for at least one of them.
The rationality of $G/H$
then follows from
lemma~\ref{orbit} or lemma~\ref{Hbu}
respectively.
\end{proof}

\bigskip

We are now in a position to show
our main result of this section.

\begin{theorem}
\label{ThB}
Let $G$ be a connected linear algebraic group,
and let $H \subseteq G$ be
a connected closed subgroup.
If $\dim(G/H) \leqslant 10$,
then $G/H$ is a rational variety.
\end{theorem}

\begin{proof}
By lemma~\ref{mod radical},
we may replace $G$
by its maximal semisimple quotient;
henceforth
we assume that
$G$ is semisimple
and adopt the notation of~\ref{notA}.
By the dimension formula~\eqref{dimformula} in~\ref{notA},
we have
\[
  \dim(G/H)
  \ =\
  (t_G - t_H) + 2(u_G - u_H) + u(H)
  .
\]
If $(t_G - t_H) + (u_G - u_H) \leqslant 5$,
the rationality of $G/H$
follows from lemma~\ref{tu6},
while if $u_G - u_H \leqslant 3$,
the rationality of $G/H$
follows from lemma~\ref{uGH3}.
In the remaining cases,
our hypothesis $\dim(G/H) \leqslant 10$
together with the above dimension formula
imply that
\[
  u(H) = 0
  ,
  \quad
  u_G - u_H = 4
  ,
  \quad
  \text{and}
  \quad
  t_G - t_H = 2
  .
\]
Hence $H$ is reductive,
and the rationality of $G/H$
now follows from lemma~\ref{uGH4}.
\end{proof}

\bigskip

With a bit more work,
we can establish
a slightly technical
but also more applicable result
in theorem~\ref{ThBrank}.
First, we note that
the dimension formula
\[
  \dim(U^- \backslash G / B_H)
  \ =\
  (t_G + u_G) - (t_H + u_H)
  ,
  \qquad
  \dim(B^- \backslash G / U_H)
  \ =\
  u_G - u_H
\]
of lemmas~\ref{tu6} and~\ref{uGH3}
yields:

\begin{lemma}
\label{equiv}
In the situation of~\ref{notA},
the following are equivalent:
\begin{itemize}
\item[(a)]
$\dim (B^- \backslash G / B_H)
 =
 u_G - u_H - t_H
$.
\item[(b)]
$\dim (U^- \backslash G / B_H)
 -
 \dim (B^- \backslash G / B_H)
 =
 t_G$
(i.e.~$d = t_G$ in lemma~\ref{BU}).
\item[(c)]
$\dim (B^- \backslash G / U_H)
 -
 \dim (B^- \backslash G / B_H)
 =
 t_H$
(i.e.~$e = t_H$ in lemma~\ref{Hbu}).
\end{itemize}
\end{lemma}

\begin{theorem}
\label{ThBrank}
Let $G$ be a connected semisimple group,
and let $H \subseteq G$ be
a connected reductive closed subgroup.
Suppose
the natural left action of $G$ on $G/H$
has zero-dimensional kernel.
If $\dim(G/H) < t_G + t_H + 8$,
then $G/H$ is a rational variety.
\end{theorem}

\begin{proof}
We adopt the notation of~\ref{notA}.
Our hypothesis on the action of $G$ on $G/H$
together with lemma~\ref{Hbu}
gives $e = t_H$,
which by lemma~\ref{equiv}
means that
$\dim(B^- \backslash G / B_H)
 =
 u_G - u_H - t_H
$.
Since $H$ is reductive by hypothesis,
one has $u(H) = 0$,
and so
by the dimension formula~\eqref{dimformula} in~\ref{notA},
we have
\[
  \dim(G/H)
  \ =\
  (t_G - t_H) + 2(u_G - u_H)
  .
\]
Hence
our assumption that
this is~$< t_G + t_H +8$
amounts to the inequality
\[
  \dim(B^- \backslash G / B_H)
  \ <\
  4
  .
\]
If $\dim(B^- \backslash G/H) \leqslant 2$,
the rationality of $G/H$
follows from lemma~\ref{dim1or2}.
In the remaining case,
by the inequality in lemma~\ref{orbit},
we must have
\[
  3
  \ \leqslant\
  \dim(B^- \backslash G / H)
  \ \leqslant\
  \dim(B^- \backslash G / B_H)
  \ \leqslant\
  3
  ,
\]
whence equality holds throughout,
and the rationality of $G/H$
follows from lemma~\ref{orbit} again.
\end{proof}

\begin{corollary}
\label{B23C3G2}
Let $G$ be a connected group
which is almost-simple
of type~$B_3$ or $G_2$,
and let $H \subseteq G$ be
a connected semisimple subgroup
of maximal rank in $G$.
Then $G/H$ is rational.
\end{corollary}

\begin{proof}
Since the case when $H = G$ is trivial,
we shall assume that
the connected semisimple closed subgroup
$H \subsetneq G$
is properly contained in $G$.
The natural left action of $G$ on $G/H$
is thus non-trivial,
and since
$G$ is almost-simple by hypothesis,
the kernel of the action
is zero-dimensional;
hence
theorem~\ref{ThBrank} is applicable
whenever
the required bound on $\dim(G/H)$ holds.
As $H$ is semisimple,
we have the crude lower bound
$\dim H \geqslant 3n$
where $n$ denotes
the common rank of $G$ and $H$.
From the following table of values:
\begin{center}
\begin{tabular}{l || c | c}
  $G$ of type
  &
  $B_3$
  &
  $G_2$
  \\
  \hline
  $\dim G = $
  &
  21
  &
  14
  \\
  $\dim H \geqslant $
  &
  9
  &
  6
  \\
  $\dim(G/H) \leqslant $
  &
  12
  &
  8
  \\
  $n + n + 8 =$
  &
  14
  &
  12
  \\
\end{tabular}
\end{center}
we see that
$\dim(G/H) < n + n + 8$
in each of the cases considered,
whence
theorem~\ref{ThBrank} yields
the rationality of $G/H$.
\end{proof}

We note that
cor.~\ref{B23C3G2}
completes the proof of
theorem~\ref{ThA}
in~\ref{ThAproof}.

%
%
%
%


\section{Concluding remarks}

In this final section,
we indicate a few other cases
in which our results yield
an affirmative answer to
the rationality problem~\ref{RatPbm}.
We still work over
an algebraically closed base field $k$
of characteristic~0.

\begin{proposition}
Let $G$ be a connected linear algebraic group
with $\dim G \leqslant 13$.
Then for any connected closed subgroup $H \subseteq G$,
$G/H$ is a rational variety.
\end{proposition}

\begin{proof}
If $\dim(G/H) \leqslant 10$,
the rationality of $G/H$
follows from theorem~\ref{ThB}.
The remaining cases are when
$11 \leqslant \dim(G/H) \leqslant 13$,
but this means
$\dim H \leqslant 2$
and thus $H$ is solvable,
whence
the rationality of $G/H$
follows from theorem~\ref{Th0}.
\end{proof}

\begin{proposition}
\label{dimG=14}
Let $G$ be a connected linear algebraic group
with $\dim G = 14$.
Then for any connected closed subgroup $H \subseteq G$,
$G/H$ is a rational variety
---
except possibly when
$G$ is the simple group of type~$G_2$
and
$H \subseteq G$ is semisimple of type~$A_1$
(in which case
$\dim(G/H) = 11$).
\end{proposition}

\begin{proof}
As in the proof of the previous result,
when $\dim H \geqslant 4$
or when $\dim H \leqslant 2$,
the rationality of $G/H$
follows from theorem~\ref{ThB} or theorem~\ref{Th0}
respectively.
Henceforth
we assume that
$\dim H = 3$ and that $H$ is not solvable.
This means that
$H$ is semisimple of type~$A_1$;
in particular,
$\dim(G/H) = 11$.

Let $R = R(G)$ be the solvable radical of $G$.
If $\dim(G/HR) < \dim(G/H) = 11$,
then $G/HR$ is rational by theorem~\ref{ThB},
and so
$G/H$ is rational by lemma~\ref{mod radical}.
Hence we may assume that
$\dim(G/HR) = \dim(G/H) = 11$,
which means
$H = HR$ is a connected closed subgroup of $G$
containing $R$ in its radical;
since $H$ is semisimple,
this forces $R$ to be trivial,
and hence $G$ is also semisimple.
By the classification of semisimple groups,
$\dim G = 14$ implies that
$G$ is either of type~$A_2+2A_1$
or of type~$G_2$.
In the former case,
we have
$u_G - u_H
 = (3 + 2) - 1 = 4
$
and the rationality of $G/H$
follows from lemma~\ref{uGH4}.
In the latter case,
we are in the possible exceptional situation
of the proposition.
\end{proof}

When the homogeneous variety $G/H$
is of dimension~$\leqslant 10$,
the rationality problem~\ref{RatPbm}
has been answered affirmatively
in theorem~\ref{ThB}.
We consider the cases when $G/H$
is of dimension~11 and~12
in prop.~\ref{dimG/H=11} and~\ref{dimG/H=12}
below.
To put the homogeneous variety $G/H$
in a somewhat ``reduced form'',
we impose the hypothesis that
$G$ acts on $G/H$
with a zero-dimensional kernel;
by lemma~\ref{tSS},
this can always be achieved
without changing the birational type of $G/H$
by replacing $G$ and $H$
by their images in $\Aut(k(G/H))$.

\begin{proposition}
\label{dimG/H=11}
Let $G$ be a connected semisimple group,
and let $H \subseteq G$ be
a connected closed subgroup.
Suppose
the natural left action of $G$ on $G/H$
has zero-dimensional kernel.
If $\dim (G/H) = 11$,
then $G/H$ is a rational variety
---
except possibly when
$G$ is semisimple of type~$G_2$
and
$H$ is semisimple of type~$A_1$.
\end{proposition}

\begin{proof}
We adopt the notation of~\ref{notA}.
If $H$ is contained in
some proper parabolic subgroup $P$ of $G$,
then by cor.~\ref{subparabolic},
$G/H$ is birational to $(G/P) \times (P/H)$,
and by lemma~\ref{parabolic section},
$G/P$ is rational.
Since $\dim(P/H) \leqslant \dim(G/H)-1 = 10$,
theorem~\ref{ThB} shows that
$P/H$ is rational,
and the rationality of $G/H$ follows.
Henceforth
we assume
$H$ is not contained in
any proper parabolic subgroup $P$ of $G$;
thus by lemma~\ref{Borel-Tits},
$H$ is semisimple.

Proceeding as in the proof of theorem~\ref{ThB},
we see that
$G/H$ is rational
if $(t_G - t_H) + (u_G - u_H) \leqslant 5$
or $u_G - u_H \leqslant 3$,
or even $u_G - u_H = 4$
(by lemma~\ref{uGH4}).
In the remaining cases,
our hypothesis $\dim(G/H) = 11$
together with the dimension formula~\eqref{dimformula} in~\ref{notA}
imply that
\[
  (\,
  u_G-u_H
  \,,\,
  t_G-t_H
  \,)
  \quad
  \text{is equal to}
  \quad
  (5,1)
  .
\]
By lemma~\ref{uGH3},
we have
$\dim(B^- \backslash G / U_H) = u_G - u_H$
which is~$= 5$ here,
so we may argue
as in the proof of lemma \ref{uGH4}
to see that
$G/H$ is rational
except possibly when
\[
  \dim (B^- \backslash G / U_H) = 5
  ,
  \quad
  \dim (B^- \backslash G / B_H) = 4
  ,
  \quad
  \dim (B^- \backslash G / H) = 3
  .
\]
In this case,
our hypothesis on the action of $G$ on $G/H$
together with lemma~\ref{Hbu}
gives $e = t_H$
in the notation there,
which by lemma~\ref{equiv}
means that
$\dim(B^- \backslash G / B_H)
 =
 u_G - u_H - t_H
$.
Thus
$H$ is a semisimple group
of rank~$t_H = 1$
and hence of type~$A_1$,
and
$G$ is a semisimple group
of rank~$t_G = t_H + 1 = 2$,
with
$\dim G = \dim(G/H) + \dim H = 14$.
By the classification of semisimple groups,
this implies
$G$ is of type~$G_2$,
and we are in the possible exceptional situation
of the proposition.
\end{proof}

\begin{proposition}
\label{dimG/H=12}
Let $G$ be a connected semisimple group,
and let $H \subseteq G$ be
a connected reductive closed subgroup.
Suppose
the natural left action of $G$ on $G/H$
has zero-dimensional kernel.
If $\dim (G/H) = 12$,
then $G/H$ is a rational variety
---
except possibly when
$G$ is semisimple of type~$A_3$
and
$H$ is semisimple of type~$A_1$.
\end{proposition}

\begin{proof}
Again
we adopt the notation of~\ref{notA},
but note that
$u(H) = 0$
by hypothesis here.
As in the proof of the previous result,
we see that
$G/H$ is rational
if $(t_G - t_H) + (u_G - u_H) \leqslant 5$
or $u_G - u_H \leqslant 4$.
In the remaining cases,
our hypothesis $\dim(G/H) = 12$
together with the dimension formula~\eqref{dimformula} in~\ref{notA}
imply that
\[
  (\,
  u_G-u_H
  \,,\,
  t_G-t_H
  \,)
  \quad
  \text{is equal to}
  \quad
  (6,0)
  \ \text{ or }\
  (5,2)
  .
\]
Our hypothesis on the action of $G$ on $G/H$
together with theorem~\ref{ThBrank}
shows that
$G/H$ is rational
if $12 = \dim (G/H) < t_G + t_H + 8$;
henceforth
we assume that
$t_G + t_H \leqslant 4$.

In the first case above,
we have
$t_G = t_H \leqslant 2$.
By the classification of semisimple groups,
this implies
$\dim G \leqslant 14$,
and hence
$\dim H = \dim G - \dim(G/H) \leqslant 2$.
This means that
$H$ is solvable,
and so
$G/H$ is rational by theorem~\ref{Th0}.

In the second case above,
we must have
$t_H = 1$
and
$t_G = t_H + 2 = 3$.
Again,
if $H$ is solvable,
then
$G/H$ is rational by theorem~\ref{Th0};
henceforth
we assume that
$H$ is not solvable.
This means that
$H$ is semisimple
of rank~$t_H = 1$
and hence of type~$A_1$,
and
$G$ is a semisimple group
of rank~$t_G = 3$,
with
$\dim G = \dim(G/H) + \dim H = 15$.
By the classification of semisimple groups,
this implies
$G$ is of type~$A_3$,
and we are in the possible exceptional situation
of the proposition.
\end{proof}

In view of the possible exceptional situations
in prop.~\ref{dimG=14} and~\ref{dimG/H=11},
our rationality results do not entirely cover
the case when
$G$ is the simple group of type~$G_2$,
and we are thus lead
to pose the following:

\begin{question}
Let $G$ be the 14-dimensional
connected semisimple group
of type~$G_2$.
Let $H \subseteq G$ be
a connected semisimple closed subgroup
of type~$A_1$
which is not contained in
a proper parabolic subgroup of $G$.
Is the $11$-dimensional homogeneous variety $G/H$
a rational variety?
\end{question}

The question can be made
a bit more precise.
By the Jacobson-Morozov theorem,
the semisimple closed subgroups
of type~$A_1$
in a semisimple group $G$
are in bijection with
the non-trivial unipotent elements in $G$.
For $G \cong G_2$,
there are four such unipotent conjugacy classes,
of which
the regular and the subregular unipotent classes
are the ones
corresponding to
($G$-conjugacy classes of)
subgroups $H$ of type~$A_1$
which are not contained in
any proper parabolic subgroup of $G$.
If $H$ corresponds to
the subregular unipotent class,
one knows that
$H \cong \PGL_2 \cong \SO_3$
is contained in
a maximal connected semisimple subgroup $M$
of type~$A_2$ in $G \cong G_2$,
and that $M \cong \SL_3$,
which is a special group
(in the sense of Serre);
by lemma \ref{spQ},
this implies that
$G/H$ is birational to $(G/M) \times (M/H)$,
and since both $G/M$ and $M/H$
are of dimension~$\leqslant 10$
and hence rational
by theorem~\ref{ThB},
it follows that
$G/H$ is also rational
in this case.
Thus
the only outstanding case of the question,
yielding the possible exceptional situation
in prop.~\ref{dimG=14} and~\ref{dimG/H=11},
is when $H$ corresponds to
the regular unipotent class of
$G \cong G_2$,
i.e.~when $H \cong \PGL_2$ arises as
the image of $\SL_2$
under its irreducible 7-dimensional representation
to $G_2 \subseteq \SO_7$.
\footnote{
Likewise,
a similar analysis
shows that
the only possible exceptional situation
in prop.~\ref{dimG/H=12}
is when the semisimple group $H$
of type~$A_1$
corresponds via the Jacobson-Morozov theorem to
the regular unipotent class of
$G$ of type~$A_3$,
i.e.~$H$ is
the isomorphic image of $\SL_2$
under its irreducible 4-dimensional representation
to $\SL_4$
---
in the ``adjoint'' case,
the homogeneous space
$\PGL_4/\PGL_2$
is known (cf.~\cite{PS})
to be rational.
}

\bigskip

We conclude by the following result
indicating the nature of
a potential ``minimal counter-example''
for the rationality of homogeneous variety $G/H$.

\begin{proposition}
Let $G$ be a connected linear algebraic group,
and let $H \subseteq G$ be
a connected closed subgroup such that
the natural left action of $G$ on $G/H$ is faithful.
Suppose that the homogeneous variety $G/H$
is not rational,
and that $\dim(G/H)$
is minimal among these non-rational varieties.
Then:
\begin{itemize}
\item[(a)]
$G$ is semisimple;
$H$ is semisimple,
and is not contained in
any proper parabolic subgroup of $G$.
\item[(b)]
$H$ is of finite index in
its normalizer $N_G(H)$ in $G$;
consequently,
the Tits fibration
$G/H \to G/N_G(H)$
is a finite morphism.
\item[(c)]
Let $X$ be
a smooth projective $G$-equivariant compactification of $G/H$
with $D := X \smallsetminus (G/H)$
a simple normal crossing divisor;
then the rational map $\Phi_{|{-}(K_X+D)|}$
given by the complete linear system $|{-}(K_X+D)|$
is a generically finite map.
In particular,
$-(K_X+D)$ is a big divisor.
\end{itemize}
\end{proposition}

\begin{proof}
Let $R := R(G)$ be the radical of $G$.
By lemma~\ref{mod radical},
$G/H$ is birational to $G'/H' \times \bP^s$
for some $s \leqslant \dim R$,
where $G' = G/R$ is
the maximal semisimple quotient of $G$,
and $H' = H/(H \cap R)$ is
the image of $H$ in $G'$.
Since $G/H$ is non-rational,
so is $G'/H'$,
whence
the minimality of $\dim(G/H)$
implies that
$s = 0$,
i.e.~$G/H$ is birational to $G'/H'$.
Since $G$ acts faithfully on $G/H$ by assumption,
while $R$ acts trivially on $G'/H'$,
it follows that $R$ is trivial and $G$ is semisimple.

Suppose $H$ is contained in
a proper parabolic subgroup $P$ of $G$.
Since $\dim(P/H)$ is strictly smaller than $\dim(G/H)$,
the minimality of $\dim(G/H)$
forces $P/H$ to be rational.
But $G/P$ is rational
by lemma~\ref{parabolic section},
while by cor.~\ref{subparabolic},
$G/H$ is birational to $(G/P) \times (P/H)$
and hence is rational,
contradicting our assumption of
the non-rationality of $G/H$.
Hence $H$ is not contained in
any proper parabolic subgroup of $G$.
By lemma~\ref{Borel-Tits},
it follows that
$H$ is semisimple.
This proves part~(a).

Suppose the closed subvariety $N_G(H)/H$ of $G/H$
has positive dimension.
Choose a connected 1-dimensional closed subgroup
$\Mbar$ in $N_G(H)/H$,
and let $M \subseteq N_G(H)$ be
its preimage in $N_G(H)$.
By lemma~\ref{Sol}
applied to the natural right action
of $\Mbar$ on $G/H$,
we see that
$G/H$ is birational to
$G/M$ or to $(G/M) \times \bP^1$.
Since $\dim(G/M)$ is strictly smaller than $\dim(G/H)$,
the minimality of $\dim(G/H)$
forces $G/M$ to be rational;
but this implies that
$G/H$ is also rational,
contradicting our assumption.
Hence $N_G(H)/H$ is 0-dimensional,
which proves part~(b).

Let $(X, D)$ be a smooth projective
$G$-equivariant compactification of
the positive-dimensional homogeneous variety $G/H$,
and let $D = X \smallsetminus (G/H)$.
As observed in \cite[\S 2.1, Proof of Prop.~3.3.5 (iii)]{Br},
the Tits fibration
$G/H \to G/N_G(H)$
(as a rational map from $X \supseteq G/H$)
is given by $\Phi_{|V|}$,
where
\[
  V
  \ :=\
  \Image
  (\,
    \wedge^{\dim G} \, (\Lie G)
    \longrightarrow
    H^0(X, -(K_X + D))
  \,)
\]
is a non-zero vector space over $k$.
The generical finiteness of $\Phi_{|V|}$
then implies the same
for $\Phi_{|{-}(K_X+D)|}$,
which proves part~(c).
\end{proof}
%
%
%
%



%
%
%
%

\providecommand{\bysame}{\leavevmode\hbox to3em{\hrulefill}}

\end{document}